\documentclass[a4paper,11pt]{amsart}
\usepackage{amsmath,amssymb,amsthm,amsfonts}
\usepackage{color}
\usepackage[all]{xy}

\usepackage{hyperref}

\newcommand{\EE}{\mathcal{E}}

\newcommand{\RR}{\mathbb{R}}

\newcommand{\ZZ}{\mathbb{Z}}
\newcommand{\NN}{\mathbb{N}}

\newcommand{\CC}{\mathbb{C}}

\newcommand{\contr}{{\mspace{1mu}\lrcorner\mspace{1.5mu}}}

\DeclareMathOperator{\Hom}{Hom}

\DeclareMathOperator{\Spec}{Spec}
\DeclareMathOperator{\rk}{rk}

\newtheorem{theorem}{Theorem}[section]
\newtheorem*{maintheorem}{Main Theorem}
\newtheorem*{Bconjecture}{Conjecture}

\newtheorem{conjecture}[theorem]{Conjecture}
\newtheorem{proposition}[theorem]{Proposition}
\newtheorem{lemma}[theorem]{Lemma}
\theoremstyle{definition}
\newtheorem{definition}[theorem]{Definition}

\newtheorem{example}[theorem]{Example}

\definecolor{rosso}{RGB}{162,0,0}
\definecolor{verde}{RGB}{0,100,0}
\definecolor{blu}{RGB}{0,0,162}

\title{On degeneracy loci of equivariant bi-vector fields on a smooth toric variety}
\author{Elena Martinengo}
\address{Dipartimento di Matematica ``Giuseppe Peano'', Universit\`a degli Studi di Torino, via Carlo Alberto 10, 10123 Torino, Italy}
\email{elena.martinengo@unito.it}


\begin{document}
\maketitle
\begin{abstract}
We study equivariant bi-vector fields on a toric variety. We prove that, on a smooth toric variety of dimension $n$, the locus where the rank of
an equivariant  bi-vector field is $\leq 2k$ is not empty and has at least a component of dimension $\geq 2k+1$,  for all integers $k> 0$ such that $2k < n$. The same is true also for $k=0$, if the toric variety is smooth and compact. While for the non compact case, the locus in question has to be assumed to be non empty.
\end{abstract}

\tableofcontents

\section*{Introduction}
Let $X$ be a smooth toric variety associated to a fan $\Sigma \subset N \otimes_\ZZ \RR$, where $N$ is an $n$-dimensional lattice. Our aim is to study degeneracy loci of equivariant bi-vector fields on $X$. A bi-vector field $\pi$ on $X$ is a $2$-form with values in the holomorphic tangent bundle $\Theta_X$. 
Let $\{ z_1,\ldots ,z_n \}$ be a system of local coordinates around a point $x\in X$, then a bi-vector field $\pi\in \Gamma(X, \bigwedge^2 \Theta_X)$ can be written locally as 
\[ \pi= \sum_{i,j=1}^n   \pi^{ij}(z_1, \ldots, z_n)  \frac{\partial}{\partial z_i}\wedge \frac{\partial }{\partial z_j},\] 
where $\Pi=\left(\pi^{ij}(z_1, \ldots, z_n)\right)_{ij}$ is an antisymmetric matrix of homolorphic functions defined locally around $x$. 
The rank of the bi-vector field $\pi$ at the point $x$ is the rank of the matrix $\Pi$ calculated in $x$. The $k$-th degeneracy locus of $\pi$ is
\[   X_{\leq k}:= \{x\in X \mid \rk_x \pi\leq  k\}.\] 
Since we are on a toric variety, it makes sense to consider bi-vector fields that are equivariant  under the action of the torus. It is obvious that their rank is constant on the orbits of the action. We prove the following
\begin{maintheorem}
Let $X$ be a smooth toric variety of dimension $n$ and $\pi \in \Gamma(X, \bigwedge^2 \Theta_X)$ an equivariant bi-vector field on $X$. 
Consider an integer $k > 0$ such that $2k < n$. Then the degeneracy locus
$X_{\leq 2k}$ is not empty and contains at least a component of dimension $\geq 2k+1$.
Moreover, if the degeneracy locus $X_{\leq 0}$ is not empty, it contains at least a curve.
If in addition $X$ is compact, then $X_{\leq 0}\neq \emptyset$. 
\end{maintheorem}
Section \ref{sec.main thm} is dedicated to the proof of this result. 
The main ingredients are the following. In Section \ref{sec.Degeneracy for equivariant}, we write the expression of an equivariant bi-vector field $\pi$ on the open dense torus $(\CC^*)^n \subset X$ and study the behaviour of it under the change of immersion $(\CC^*)^n \hookrightarrow Y_\sigma$ of the torus on the affine toric subvariety $Y_\sigma \subset X$, when varying the maximal cone $\sigma$ in the fan $\Sigma$. 
Looking at the affine charts, we are able to prove that, for all integers $k > 0$ such that $2k < n$, there exists a $(2k+1)$-dimensional orbit of the action of the torus on $X$ contained in the degeneracy locus $X_{\leq 2k}$. For $k=0$, the compactness of $X$ is needed to prove with the same arguments the existence of a $1$-dimensional orbit contained in $X_{\leq 0}$. While for $X$ non compact, we have to assume $X_{\leq 0}\neq \emptyset$ to conclude. 

\medskip 

The idea to investigate the degeneracy loci is not new. In \cite{Bondal}, Bondal studied them for Poisson structures and formulated the following

\begin{Bconjecture}
 Let $X$ be a connected Fano variety of dimension $n$ and $\pi$ a Poisson structure on it. Consider an integer $k \geq 0$ such that $2k < n$. If the degeneracy locus
$X_{\leq 2k}$ is not empty, it contains a component of dimension $\geq 2k+1$.
\end{Bconjecture}
Recall that a Poisson structure on a smooth complex variety $X$ is a bi-vector field $\pi \in  \Gamma(X, \bigwedge^2 \Theta_X)$ such that the Poisson bracket 
\[ \{f,g\} := \pi 	\contr (df \wedge dg ) \]
satisfies the Jacobi identity on holomorphic functions.

Later in \cite{Beauville}, Beauville conjectured the same property to be true for Poisson structures on smooth compact complex varieties. 
In \cite{Polishchuk}, Polishchuk proved it for the maximal non trivial degeneracy locus in two cases: when $X$ is the projective space and when $\pi$ has maximal rank at the general point.  

The interest of these authors for Poisson structures comes from different areas.
One is the problem of classification of quadratic Poisson structures, i.e. Poisson brackets on the algebra of polynomials in $n$ variables such that the brackets are quadratic forms. These structures arise as tangent spaces to noncommutative deformations of the polynomial algebra itself. 
On the other hand, Poisson manifolds, i.e. smooth complex varieties with a Poisson structure, provide a more flexible notion then the one of hyperk\"ahler manifolds. 

Bondal's conjecture on a toric variety was investigated for the first time by Gay in \cite{Gay}. He proved it for invariant Poisson structures on a smooth toric variety and showed that, in this case, compactness is not necessary. 

The next natural step is to study equivariant Poisson structures, or more generally equivariant bi-vector fields, on a toric variety, that is the object of our article.

\subsection*{Acknowledgements}
I thank Marco Manetti for introducing me to the problem and for some very helpful discussions and suggestions. I thank Andreas Hochenegger for reading the first version of the article.
 
\section{Toric varieties}
In this section we summarise some basic notions on toric variety, following mainly the book \cite{CLS}. We also recall the classical references \cite{Danilov, Fulton, Oda}. 
\subsection{Notations}
Let $\CC^*$ be the multiplicative group of non zero elements of $\CC$. A \emph{complex algebraic torus} (or simply a \emph{torus}) is an affine variety $T$ isomorphic to the group $(\CC^*)^n$ for some $n>0$.
There are two \emph{lattices}, i.e. free abelian groups, of rank $n$ associated to a torus $T= \Spec \CC[t_1^{\pm 1}, \ldots, t_n^{\pm 1}]$: 
\begin{itemize}
\item the lattice of characters of  $T$:
\begin{eqnarray*}
M &=& \left\{\chi: T \to \CC^* \mid \chi \ \mbox{group homomorphism} \right\}\\ & =& \left\{\chi^m:T \to \CC^*, \chi^m(t_1, \ldots , t_n)=t_1^{m_1}\cdots t_n^{m_n} \mid m=(m_1,\ldots, m_n) \in \ZZ^n\right\};\end{eqnarray*}
\item the lattice of $1$-parameter subgroups of $T$:
\begin{eqnarray*}
N&=& \left\{\lambda: \CC^* \to T \mid \lambda \ \mbox{group homomorphism} \right\} \\
&=& \left\{ \lambda_u : \CC^* \to T, \lambda_u(t)=( t^{u_1}, \ldots, t^{u_n}) \mid u=(u_1, \ldots, u_n) \in \ZZ^n   \right\}.\end{eqnarray*}
\end{itemize} 
There is a natural pairing  $\langle -, - \rangle: M \times N \to \ZZ$, that associates to every $(m,u)$ the integer associated to the character $ \chi^m \circ \lambda_u$. It follows that $M$ is the dual lattice of $N$ and viceversa. Moreover, there is a canonical isomorphism $N\otimes_\ZZ \CC \cong T$, given by $u\otimes t \mapsto \lambda^u(t)$, and the notation $T_N$  is used to indicate the torus associated to the lattice $N$ via this isomorphism.
To fix notation: let $N_\RR:=N\otimes_\ZZ \RR$ and $M_\RR:= M \otimes_\ZZ \RR$.

\subsection{Toric variety associated to a fan}
\begin{definition}
A \emph{toric variety} is an irreducible variety $X$ containing a torus $T\cong (\CC^*)^n$ as a dense open subset, such that the action of $T$ on itself extends to an action on $X$. 
\end{definition}
Given a cone, one can define an affine toric variety in the following way.
\begin{definition}
A \emph{convex polyhedral cone} in $N_\RR$ is a set of the form
\[ \sigma =\mbox{Cone}(S) := \left\{ \sum_{s\in S} \lambda_s s \mid \lambda_s\geq 0 \right\} \subset N_\RR\]
where $S\subset N_\RR$ is a set of \emph{generators} of $\sigma$. 
\end{definition}
Note that a convex polyhedral cone is convex and it is a cone. In the following we refer to such cones simply with the name \emph{polyhedral cones}.

Given a polyhedral cone $\sigma$, its \emph{dual cone} is
\[\sigma^\vee :=\{ m \in M_\RR \mid \langle m,u \rangle\geq 0, \ \mbox{for all} \ u \in \sigma \},  \]
\[ \mbox{while} \quad \sigma^\perp :=\{ m \in M_\RR \mid \langle m,u \rangle=0, \ \mbox{for all} \ u \in \sigma \}, \]
that are both polyhedral cones.
A \emph{face} of a polyhedral cone $\sigma$ is the intersection of the cone with an hyperplane $H_m=\{u \in N_\RR \mid \langle m, u \rangle =0 \}$ for some $m \in \sigma^\vee$. One writes $\tau \leq \sigma$ to indicate that $\tau$ is a face of $\sigma$. A face of dimension one is also called an \emph{edge}.
 
\begin{definition}
Let $\sigma$ be a polyhedral cone. It is called
\begin{itemize}
\item \emph{strongly convex} if the origin is a face of $\sigma$;
\item \emph{rational} if $\sigma =\mbox{Cone}(S)$, for some finite set  $S\subset N_\RR$;
\end{itemize}\end{definition}
A strongly convex rational polyhedral cone $\sigma$ has a canonical set of generators constructed as follow. Let $\rho$ be an edge of $\sigma$. Since $\sigma$ is strongly convex, $\rho$ is a half line, and since $\sigma$ is rational, the semigroup $\rho \cap N$ is generated by a unique element $u_\rho$, called a \emph{ray generator}. The ray generators of the edges of $\sigma$ are called  the \emph{minimal generators}.

\begin{definition}
A rational polyhedral cone $\sigma$ is called \emph{smooth} if the set of its minimal generators is a part of a $\ZZ$-basis of $N$. 
\end{definition}

The following theorem defines the affine toric variety associated to a cone.

\begin{theorem}
Let $\sigma\subset N_\RR$ be a rational polyhedral cone.
The lattice points $S_\sigma := \sigma^\vee \cap M$ form a finitely generated semigroup. 
Let $\CC[S_\sigma]$ be the group algebra associated to $S_\sigma$.
Then
\[ Y_\sigma  := \Spec\left( \CC[S_\sigma]\right) \]
is an affine toric variety. 
Moreover, 
\begin{itemize}
\item $\dim Y_\sigma = \rk_\ZZ N \Leftrightarrow  \sigma$ is strongly convex, 
\item $Y_\sigma$ is normal $\Leftrightarrow \sigma$ is strongly convex and rational,  
\item $Y_\sigma$ is smooth $\Leftrightarrow$ $\sigma$ is smooth.
\end{itemize}
\end{theorem}
An important fact is that all affine toric varieties can be obtained in this way.
\begin{example}
The $n$-dimensional torus $\left(\CC^*\right)^n\subset Y_\sigma$ is the subvariety associated to the $0$-dimensional cone $\{0\}$. Indeed, $\left(\CC^*\right)^n= \Spec \CC[\chi^{\pm e^*_1}, \ldots, \chi^{\pm e^*_n}] = \Spec\CC[z_1^{\pm 1}, \ldots z_n^{\pm 1}]$, where $\EE:=\{e_1, \ldots, e_n\}$ is the standard basis of $\ZZ^n$.

\end{example}

We will now associate a toric variety to a fan.
 
\begin{definition}
A \emph{fan} $\Sigma:=\{ \sigma_i\}_{i\in I}$ is a finite collection of strongly convex rational polyhedral cones  
in $N_\RR$ such that
\begin{itemize}
\item every face of a cone $\sigma \in \Sigma$ is a cone of $\Sigma$,
\item for every $\sigma_1, \sigma_2 \in \Sigma$ the intersection $\sigma_1 \cap \sigma_2$ is a common face of $\sigma_1$ and $\sigma_2$. 
\end{itemize}
\end{definition}

\begin{definition}
A fan $\sigma\subset N_\RR$ is called 
\begin{itemize}
\item \emph{smooth} if every cone of $\Sigma$ is smooth;
\item \emph{complete} if its support $|\Sigma| =\bigcup_{\sigma \in \Sigma} \sigma$ is all $N_\RR$. 
\end{itemize}
\end{definition}

As already seen, every cone $\sigma$ defines an affine toric variety. Moreover, it turns out that the cones in the fan give the combinatorial data necessary to glue a collection of affine toric varieties together to yield an abstract toric variety denoted by $X:=TV(\Sigma)$. The main result is the following
\begin{theorem}
Let $\Sigma$ be a fan in $N_\RR$. The variety $X=TV(\Sigma)$ is a normal separated toric variety. Moreover, 
\begin{itemize}
\item $X$ is smooth $\Leftrightarrow$ $\Sigma$ is smooth,
\item $X$ is compact in the classical topology $\Leftrightarrow$ $\Sigma$ is complete. 
\end{itemize}
\end{theorem}
An important result is
\begin{theorem}
Let $X$ be a normal separated toric variety with torus $T_N$. Then there exists a fan $\Sigma \subset N_{\RR}$ such that $X\cong TV(\Sigma)$.
\end{theorem}

For future use, we prove the following
\begin{lemma}\label{lemma.cono}
Let $\sigma\subset N_\RR$ be a strongly convex polyhedral cone.
Consider the hyperplane 
\[ H_0 =\{ (z_1, \ldots, z_n) \mid a_1z_1 +\ldots +a_n z_n=0\} \subset N_\RR, \quad \mbox{with } a_i \in \RR\]
and the semispaces 
\[ H_0^{+} =\{ (z_1, \ldots, z_n) \mid a_1z_1 +\ldots +a_n z_n\geq 0\} \mbox{ and } H_0^{-} =\{ (z_1, \ldots, z_n) \mid a_1z_1 +\ldots +a_n z_n\leq 0\}. \]
If $\sigma \cap H_0^-\neq \{0\}$, there exists at least one generating ray of $\sigma$ contained in $H_0^-$. 
The same holds for $H_0^+$.  
\end{lemma}
\begin{proof} 
Let $0\neq p \in H_0^-\cap \sigma$, thus 
\[  p = \sum_{s \in S} \lambda_s \rho_s, \quad \mbox{where } \{\rho_s \}_{s\in S} \ \mbox{is a set of generating rays and} \  \lambda_s \geq 0.\] 
If $\rho_s \subset H_0^+$ for all $s\in S$, then $p \in H_0^+$ and this contradicts the assumptions.
Indeed, let $H_0=\{ (z_1, \ldots, z_n) \mid a_1z_1 +\ldots +a_n z_n=0\}$ and let $\rho_s = (r_{1s}, \ldots, r_{ns})$, then
\[ p =   \sum_{s \in S} \lambda_s \rho_s  =  \sum_{s \in S} \lambda_s  (r_{1s}, \ldots, r_{ns})=\left( \sum_{s \in S} \lambda_s r_{1s}, \ldots,  \sum_{s \in S} \lambda_s r_{ns} \right) \quad \mbox{and}\]
\[ a_1  \sum_{s \in S} \lambda_s r_{1s} +\ldots + a_n  \sum_{s \in S} \lambda_s r_{ns} =  \sum_{s \in S} \lambda_s (a_1r_{1s} +\ldots + a_n r_{ns}) \geq 0. \] 
\end{proof}

\subsection{Orbit-cone correspondence} \label{sec.orbit-cone}
Let $N$ be a $n$-dimensional lattice and $M$ its dual lattice.
Let $X=TV(\Sigma)$ be the toric variety associated to a fan $\Sigma\subset N_\RR$. Let $\sigma \in \Sigma$ be a cone and $Y_\sigma$  be the affine toric subvariety associated to it. 

Recall that, the torus $T_N =(\CC^*)^n$ acts on $Y_\sigma$ as follows. 
Let $\{m_1, \ldots, m_s\}\subset M\cong \ZZ^n$ be the characters that generate the semigroup $S_\sigma$,  $\underline{t}=(t_1, \ldots, t_n)\in (\CC^*)^n$ and  $\underline{y}=(y_1, \ldots , y_s) \in Y_\sigma$, then
\[ (\CC^*)^n \times Y_\sigma \to Y_\sigma, \quad \underline{t}\cdot \underline{y 
}:= \left(\chi^{m_1}(\underline{t})\cdot y_1, \ldots , \chi^{m_s}(\underline{t})\cdot y_s\right). \] 
To determine the orbits for this action, it is convenient to use the following intrinsic description of points of $Y_\sigma$.
Each point  $p\in Y_\sigma$ can be seen as  a semigroups homomorphisms $\gamma_p: S_\sigma \to \CC$, given by $\gamma_p(m) =\chi^m(p)$. 
The action of the torus $T_N$ on $Y_\sigma$ translates into
$ t \cdot \gamma_p (m) = \chi^m(t) \cdot \chi^m(p)$.
For every cone $\sigma\in \Sigma$, we define the \emph{distinguish point} of $\sigma$ as
the point in $Y_\sigma$ associated to the homomorphism
\[ \gamma_\sigma: S_\sigma \to \CC, \quad \mbox{given by } \gamma_\sigma(m) =: \left\{ \begin{array}{cc} 1 & \mbox{if } m \in \sigma^\perp \cap M, \\ 0 & \mbox{otherwise.} \end{array} \right. \] 
For every cone $\sigma \in \Sigma$, consider
\[ O(\sigma):= \{\gamma: S_\sigma \to \CC \mid \gamma(m)\neq 0 \mbox{ for } m \in \sigma^\perp \cap M \mbox{ and } \gamma(m)= 0  \mbox{ otherwise}\},\]
it is easy to see that it is invariant under the torus action,  $\gamma_\sigma \in O(\sigma)$ and that $O(\sigma) =T_N \cdot \gamma_\sigma$.   
The main theorem is the following
\begin{theorem}
Let $X=TV(\Sigma)$ be the toric variety associated to a fan $\Sigma$ in $N_\RR$. Then:
\begin{itemize}
\item There is a bijective correspondence 
\[\xymatrix@R=1ex{ \{\mbox{cones\  }\sigma \ \mbox{in} \ \Sigma \}  \ar@{<->}[r] & \{ T_N- \mbox{orbits in}\ \Sigma \},\\
\sigma   \ar@{<->}[r] & O(\sigma) \cong \Hom(\sigma^{\perp} \cap M, \CC^*)} \]
\item Let $n = \dim N_\RR$. For each cone $\sigma \in \Sigma$, $\dim O(\sigma) = n-\dim \sigma$.
\item The affine open subset $Y_\sigma$ is the union of orbits
\[  Y_\sigma = \bigcup_{\tau \leq \sigma} O(\tau).\]
\item $\tau \leq \sigma$ if and only if $O(\sigma) \subset \overline{O(\tau)}$ and $ \overline{O(\tau)} = \bigcup_{\tau \leq \sigma} O(\sigma)$.
\end{itemize}
\end{theorem}
To conclude the section, let's write explicitly the orbit $O(\tau)$ in $Y_\sigma$ for a face $\tau$ of a smooth cone  $\sigma$. Since $\sigma$ is smooth, we can assume it is generated by the standard basis $\{ e_1, \ldots e_n\}$ of $\ZZ^n$. Take $\tau$ to be the face generated by $\{ e_1, \ldots, e_h\}$, with $ 1 \leq h\leq n$. 
By definition, the distinguish point $\gamma_\tau$ is the semigroup homomorphism
\[ \gamma_\tau:S_\sigma \to \CC, \quad \gamma_\tau(e_i^*) =: \left\{ \begin{array}{cl} 0 & \mbox{for \ } 1\leq i \leq h, \\ 1 & \mbox{for \ } h+1 \leq i \leq n. \end{array} \right.  \]
Thus, the orbit is
\[ O(\tau) = T_N \cdot \gamma_\tau = \{(z_1, \ldots, z_n) \in Y_\sigma \mid z_i =0 \mbox{ for } 1\leq i \leq h \mbox{ and }   z_i \neq 0 \mbox{ for } h+1\leq i \leq n \}. \]

\section{Degeneracy loci of equivariant bi-vector fields} \label{sec.Degeneracy for equivariant}
\subsection{Equivariant bi-vector fields} 
Fix a $n$-dimensional lattice $N$, its dual lattice $M$ and a smooth  fan $\Sigma\subset N_\RR$.
Let $X=TV(\Sigma)$ be the smooth toric variety associated to the fan $\Sigma$.

Our aim is to study \emph{bi-vector fields} on the toric variety $X$, i.e. $2$-forms on  $X$ with values in the holomorphic tangent bundle $\Theta_X$. 
Let $\{ z_1,\ldots ,z_n \}$ be a system of local coordinates around a point $x\in X$, then $\pi$ can be written locally as 
\[ \pi= \sum_{i,j=1}^n   \pi_{ij}(z_1, \ldots, z_n)  \frac{\partial}{\partial z_i}\wedge \frac{\partial }{\partial z_j},\] 
where $\Pi=\left(\pi_{ij}(z_1, \ldots, z_n)\right)_{ij}$ is an antisymmetric matrix of homolorphic functions defined around $x\in X$.
The \emph{rank} of the bi-vector field $\pi$ at the point $x$ is the rank of the matrix $\Pi$ when calculated in $x$, we indicate it with $\rk_x \pi$. 

In particular, we are interested in equivariant bi-vector fields. 
\begin{definition} A bi-vector field $\pi \in \Gamma(X, \bigwedge^2 \Theta_X)$ is \emph{equivariant} under the torus action if there exists a character $\underline{\alpha}=(\alpha_1, \ldots, \alpha_n) \in M$, such that 
\[\pi({\underline{t}} \cdot {\underline{x}}) = {\underline{t}}^{\underline{\alpha}} \cdot \pi({\underline{x}}), \quad \mbox{for all } {\underline{x}}\in X \mbox{ and } {\underline{t}} \in (\CC^*)^n,  \] 
where $ {\underline{t}}^{\underline{\alpha}}= t_1^{\alpha_1}\cdots t_n^{\alpha_n}$. The character $\underline{\alpha}$ is called \emph{multi-degree}.
If  $\underline{\alpha}=(1, \ldots, 1)$, then the bi-vector field is \emph{invariant}.
\end{definition}

\begin{lemma}
The rank of an equivariant bi-vector field  $\pi$ is constant on the orbit of the torus action.
\end{lemma}
\begin{proof}
It is obvious, since the $\pi$ is an equivariant form.
\end{proof}

First we describe equivariant bi-vector fields on the open dense torus $(\CC^*)^n \subset X$. Let $\underline{z}=(z_1, \ldots , z_n)$ be a system of coordinates on $(\CC^*)^n$ and $\Theta$ the holomorphic tangent bundle of $(\CC^*)^n.$
 
\begin{lemma}
With the above notations,  a bi-vector field $\pi \in \Gamma ((\CC^*)^n, \bigwedge^2\Theta)$ is equivariant of multi-degree $\underline{\alpha}=(\alpha_1, \ldots, \alpha_n)$ if and only if it has the form
\begin{equation}\label{eq.forme equivarianti} \pi(\underline{z})= \sum_{i,j=1}^n a_{ij} z_1^{\alpha_1} \cdots z_i^{\alpha_i+1}\cdots  z_j^{\alpha_j+1} \cdots  z_n^{\alpha_n} \frac{\partial}{\partial z_i}\wedge \frac{\partial }{\partial z_j} =  \sum_{i,j=1}^n a_{ij} \cdot {\underline{z}}^{\underline{\alpha}}\cdot z_i\cdot  z_j \frac{\partial}{\partial z_i}\wedge \frac{\partial }{\partial z_j},  \end{equation}
where $A=(a_{ij})_{ij}$ is an antisymmetric matrix with coefficients in $\CC$ and
${\underline{z}}^{\underline{\alpha}} = z_1^{\alpha_1}\cdots z_n^{\alpha_n}$.
\end{lemma}
\begin{proof}
First check such a $\pi$ is equivariant. Recall that the action of the torus on itself is just the multiplication. For all $\underline{z}, \underline{t} \in (\CC^*)^n$, 
\begin{eqnarray*} \pi(\underline{t} \cdot \underline{z}) & = &  \sum_{i,j=1}^n a_{ij} \cdot \underline{t}^{\underline{\alpha}} {\underline{z}}^{\underline{\alpha}}\cdot t_i z_i\cdot  t_j z_j \frac{1}{t_i}\frac{\partial}{\partial z_i}\wedge \frac{1}{t_j}\frac{\partial }{\partial z_j} \\
&=& \sum_{i,j=1}^n a_{ij} \cdot \underline{t}^{\underline{\alpha}} {\underline{z}}^{\underline{\alpha}}\cdot z_i\cdot z_j \frac{\partial}{\partial z_i}\wedge\frac{\partial }{\partial z_j} = \underline{t}^{\underline{\alpha}}\cdot \pi(\underline{z}).  
 \end{eqnarray*}
Moreover, the calculation shows that the only polynomials that are equivariant under the torus action are the monomials of a fixed multidegree $\underline{\alpha}$. 
\end{proof}
Next we study the behaviour of the form $\pi$ under the change of immersion $(\CC^*)^n\hookrightarrow Y_\sigma$ of the dense torus in the affine toric subvariety $Y_\sigma \subset X$, when varying the maximal cone $\sigma$ in the 
fan $\Sigma$.

Let $X=TV(\Sigma)$ be a smooth toric variety and $\pi \in \Gamma(X,\bigwedge^2 \Theta_X)$ an equivariant bi-vector field on $X$.  Let $\sigma_0$  and $\sigma_1$ be two maximal cones of $\Sigma$. Since $\Sigma$ is a smooth fan, every maximal cone is generated by a $\ZZ$-basis of $N$. Without loss of generality,  we can assume the standard basis $\EE=\{ e_1, \ldots, e_n\}$ of $\ZZ^n$ to be a system of generators of the cone $\sigma_0$ and the form $\pi$ to be expressed on the toric affine subvariety $Y_{\sigma_0}$ by (\ref{eq.forme equivarianti}), where $z_i :=\chi^{e_i}$ are the coordinates on $Y_{\sigma_0}$. 
Let $\mathcal{V}:=\{v_1,\ldots,v_n \}$ be a basis of $\ZZ^n$ that generates the cone $\sigma_1$ and $w_i := \chi^{v_i}$ the corresponding coordinates on $Y_{\sigma_1}$.
The form $\pi$ is described on $Y_{\sigma_1}$ by the following

\begin{lemma}
 \label{lemma.cambio di coordinate}
Let $X=TV(\Sigma)$ be a smooth toric variety and $\pi \in \Gamma(X,\bigwedge^2 \Theta_X)$ an equivariant bi-vector field on $X$. In the above notations, $\pi$ can be expressed on $Y_{\sigma_1}$ by 
\begin{equation}\label{eq.forme eq base change}
\pi=\sum_{k,h=1}^n  b_{kh} \cdot  {\underline{w}}^{\underline{\beta}}\cdot  w_k \cdot w_h \frac{\partial}{\partial w_k}\wedge \frac{\partial }{\partial w_h},
  \end{equation}
with $ {\underline{w}}^{\underline{\beta}}=w_1^{\beta_1} \cdots w_n^{\beta_n} $, $B=(b_{kh})_{kh}=S \cdot A \cdot S^t $ and $\underline{\beta} = R^t  \cdot  \underline{\alpha}$, where $R$ is the matrix of base change between the system of generators $\mathcal{V}$ of $\sigma_1$ and the standard basis $\EE$ that generates $\sigma_0$ and $S=R^{-1}$.  
\end{lemma}
\begin{proof} In the above notations, $R:=(r_{ij})_{ij}$ is the matrix of base change from  $\mathcal{V}$ to $\EE$ and $S:=R^{-1}=(s_{ij})_{ij}$ is the matrix of the base change from $\mathcal{E}$ to $\mathcal{V}$, i.e.
\[  v_j = \sum_{i=1}^n r_{ij}e_i  \quad \mbox{ and }\quad e_j =  \sum_{i=1}^n s_{ij}v_i.\]

The dual basis ${\mathcal{V}}^*:=\{v_1^*, \ldots, v_n^*\}$ is a system of generators of the dual cone $\sigma_1^\vee$. The matrix of the base change from ${\mathcal{V}}^*$ to the dual of the standard basis $\EE^*$  is $(R^{1})^{t}$ and the matrix of the base change from $\EE^*$ to ${\mathcal{V}}^*$  is $R^t$, i.e.
\[ v_i^* = \sum_{j=1}^n s_{ij} e_j^* \quad \mbox{ and } \quad  e_i^* = \sum_{j=1}^n r_{ij} v_j^*.\]
The coordinates on the affine subvariety $Y_{\sigma_1}$ are by definition $ w_i:=\chi^{v_i^*}$.  
Thus, 
\begin{eqnarray*} w_i&:=& \chi^{v_i^*} = \chi^{\sum_{j=1}^n s_{ij} e_j^*} = \chi^{s_{i1}e_1^*}\cdot \ldots \cdot \chi^{s_{in}e_n^*} = z_1^{s_{i1}}\cdot \ldots \cdot z_n^{s_{in}} \quad \mbox{and} \\
 z_i&:=& \chi^{e_i^*} =\chi^{\sum_{j=1}^n r_{ij} v_j^*} = \chi^{r_{i1}v_1^*} \cdot \ldots \cdot \chi^{r_{in}v_n^*}= w_1^{r_{i1}}\cdot\ldots\cdot w_n^{r_{in}}. \end{eqnarray*}
Moreover
\[ z_j \cdot \frac{\partial}{\partial z_j} = z_j \cdot \sum_{i=1}^n \frac{\partial w_i}{\partial z_j}\frac{\partial}{\partial w_i} =z_j  \cdot\sum_{i=1}^n s_{ij} \cdot z_1^{s_{i1}}\cdots z_j^{s_{ij}-1} \cdots z_n^{s_{in}} \frac{\partial}{\partial w_i} = \sum_{i=1}^n s_{ij} w_i \frac{\partial}{\partial w_i}. \]
On the affine subvariety $Y_{\sigma_1}$, the espression of the $2$-form $\pi$  described in (\ref{eq.forme equivarianti}) becomes 
\begin{eqnarray*}   
\pi&=&  \sum_{i,j=1}^n a_{ij} \cdot z_1^{\alpha_1}\cdots  z_{n}^{\alpha_n}\cdot  z_i \frac{\partial}{\partial z_i}\wedge z_j \frac{\partial }{\partial z_j} =\\
&=& \sum_{i,j=1}^n a_{ij} \cdot z_1^{\alpha_1}\cdots z_{n}^{\alpha_n}\cdot  \left(\sum_{k=1}^n s_{ki} w_k \frac{\partial}{\partial w_k}\right)\wedge\left(\sum_{h=1}^n s_{hj} w_h \frac{\partial }{\partial w_h} \right) =\\
&=& \sum_{i,j=1}^n a_{ij} \left(w_1^{r_{11}}\cdots w_n^{r_{1n}}\right)^{\alpha_1}  \cdots \left(w_1^{r_{n1}}\cdots w_n^{r_{nn}}\right)^{\alpha_n} \left(\sum_{k=1}^n s_{ki} w_k \frac{\partial}{\partial w_k}\right)\wedge\left(\sum_{h=1}^n s_{hj} w_h \frac{\partial }{\partial w_h} \right) =\\
&=&  \sum_{k,h=1}^n \sum_{i,j=1}^n s_{ki} \cdot a_{ij} \cdot s_{hj} \cdot  {\underline{w}}^{\underline{\beta}}\cdot  w_k \frac{\partial}{\partial w_k}\wedge w_h \frac{\partial }{\partial w_h} =\\
&=&  \sum_{k,h=1}^n  b_{kh} \cdot  {\underline{w}}^{\underline{\beta}}\cdot  w_k \frac{\partial}{\partial w_k}\wedge w_h \frac{\partial }{\partial w_h},\\
  \end{eqnarray*}
where 
\[  B=(b_{kh})_{kh}=S \cdot A \cdot S^t \quad \mbox{and} \quad   \underline{\beta}=\left(\begin{array}{c}\beta_1\\ \vdots \\ \beta_n \end{array} \right) = \left(\begin{array}{ccc}r_{11} & \ldots & r_{n1}  \\  \vdots & \vdots & \vdots \\ r_{1n} & \ldots & r_{nn}\end{array} \right)\cdot \left(\begin{array}{c}\alpha_1\\ \vdots \\ \alpha_n \end{array} \right) = R^t  \cdot  \underline{\alpha}. \]
\end{proof}

\subsection{Degeneracy loci} \label{sub.sect degeneracy loci}
In this subsection we leave the toric setting and consider a smooth complex variety $X$ of dimension $n$ with a bi-vector field $\pi\in \Gamma(X, \bigwedge^2 \Theta_X)$. 
There is a partition 
\[ X = \coprod_{s \in 2\NN}  X_s, \quad \mbox{where } X_s := \{x\in X \mid \rk_x \pi = s\},\]
where the sets $X_s$ will not be subvarieties of X in general. 
While 
\[   X_{\leq s}:= \{x\in X \mid \rk_x \pi\leq  s\}= \coprod_{s\leq r} X_r\] 
is a closed subvariety of $X$ and is called the \emph{$s$-th degeneracy locus} of $\pi$.

The following Section \ref{sec.main thm} is dedicated to the proof of our
\begin{maintheorem}  
Let $X$ be a smooth toric variety of dimension $n$ and $\pi \in \Gamma(X, \bigwedge^2 \Theta_X)$ an equivariant bi-vector field on $X$. 
Consider an integer $k > 0$ such that $2k < n$. Then the degeneracy locus
$X_{\leq 2k}$ is not empty and contains at least a component of dimension $\geq 2k+1$.
Moreover, if the degeneracy locus $X_{\leq 0}$ is not empty, it contains at least a curve.
If in addition $X$ is compact, then $X_{\leq 0}\neq \emptyset$. 
\end{maintheorem}
We now explain where the interest for degeneracy loci comes from, in which cases they are known to be not empty and to contain components of a required dimension.
In \cite{Bondal}, Bondal formulated the following
\begin{conjecture}
Let $X$ be a connected Fano variety and $\pi$ a Poisson structure on it. Let  $k \geq 0$ be an integer, such that $2k < n$. Then, if the subvariety 
$X_{\leq 2k}$ is not empty, it contains at least a component of dimension $\geq 2k+1$.
\end{conjecture}
We recall that, a \emph{Fano variety} is a compact complex variety whose anticanonical bundle  is ample.

A \emph{Poisson structure} on a smooth complex variety $X$ is a bi-vector field $\pi \in  \Gamma(X, \bigwedge^2 \Theta_X)$ such that the Poisson bracket 
\[ \{f,g\} := \pi 	\contr (df \wedge dg ) \]
satisfies the Jacobi identity on holomorphic functions. An equivalent condition for a bi-vector field $\pi$ to be a Poisson structure is to satisfy $[\pi, \pi]_{SN}=0$, where $[-,-]_{SN}$ is the Schouten-Nijenhuis bracket (see \cite{Polishchuk}). More explicitly, let $\{z_1,\ldots, z_n\}$ be a system of local coordinates around a point $x\in X$ and 
\[\pi= \sum_{i,j=1}^n \pi_{ij} \frac{\partial}{\partial z_i}\wedge \frac{\partial}{\partial z_j}\]
be a bi-vector field.
Then, $[\pi, \pi]_{SN}=0$ if and only if
\begin{equation}\label{eq.condizione Poisson}\sum_{i=1}^n \pi_{ij}  \frac{\partial \pi_{hk}}{\partial z_i} + \pi_{ih}  \frac{\partial \pi_{kj}}{\partial z_i} + \pi_{ik}  \frac{\partial \pi_{jh}}{\partial z_i} =0, \quad \mbox{for all } j,h,k. \end{equation}
 
In \cite{Beauville}, Beauville stated the following more optimistic version of Bondal's conjecture.
\begin{conjecture}
Let $X$ be a smooth complex and compact variety and $\pi$ a Poisson structure on it. Let  $k \geq 0$ be an integer, such that $2k < n$. Then, if the subvariety 
$X_{\leq 2k}$ is not empty, it contains at least a component of dimension $\geq 2k+1$.
\end{conjecture}
As explained in \cite{Beauville}, there are some arguments in favor of this conjecture. 
\begin{proposition}
Let $X$ be a smooth compact complex variety and $\pi$ a Poisson structure on it. Then, 
\begin{itemize}
\item Every component of $X_s$ has dimension $\geq s$.
\item Let $r$ be the generic rank of $\pi$. Assume that $c_1(X)^q \neq 0$ in $H^q(X, \Omega_X^q)$, where $q=\dim X - r +1$. Then, the degeneracy locus $X\setminus X_r$ has a component of dimension $>r-2$.
\item Let $X$ be a projective threefold. If $X_0$ is not empty, it contains a curve. 
\end{itemize}
\end{proposition}

In \cite{Polishchuk}, Polishchuk proved Bondal's conjecture for the
maximal nontrivial degeneracy locus in two cases: when $X$ is the projective space and when the Poisson
structure $\pi$ has maximal possible rank at the general point.

In \cite{Gay}, Gay studied invariant Poissons structures on a smooth toric variety. In this setting he proved Bondal's conjecture and pointed out that compactness is not necessary.  

We will study equivariant bi-vector fields on smooth toric varieties. Our result is a little stronger then the one conjectured by Bondal. For $k \neq 0$,  we even prove the degeneracy loci to be not empty, while for $k=0$ the non emptiness is true just for compact toric varieties.
Even if we will consider all equivariant bi-vector fields on a toric variety $X$, we spell out when such a bi-vector field is a Poisson structure.
\begin{lemma}
Let $X$ be a toric variety and $\pi \in \Gamma(X, \bigwedge^2 \Theta_X)$ an equivariant bi-vector field described on the open dense torus $(\CC^*)^n$ as in  
(\ref{eq.forme equivarianti}). It defines a Poisson structure on $X$ if and only if
\[  \sum_{\scriptsize{\begin{array}{c}i=1 \\ i\neq j,h,k\end{array}}}^n  \alpha_i   \left(a_{ij} a_{hk} + a_{ih} a_{kj}  + a_{ik} a_{jh} \right)=0 \quad \mbox{for all } j\neq h \neq k. \]
\end{lemma}

\begin{proof}
By Lemma \ref{lemma.cambio di coordinate}, the structure of $\pi$ is invariant under change of affine chart. Thus, it is enough to prove the proposition on the open dense torus  $(\CC^*)^n$ on which $\pi$ is described as in (\ref{eq.forme equivarianti}). 
Because of equation (\ref{eq.condizione Poisson}), $\pi$ is a Poisson structure if and only if, for all $j,h,k$
\begin{eqnarray*}
0&=&\sum_{i=1}^n \pi_{ij}  \frac{\partial \pi_{hk}}{\partial z_i} + \pi_{ih}  \frac{\partial \pi_{kj}}{\partial z_i} + \pi_{ik}  \frac{\partial \pi_{jh}}{\partial z_i} =\\ 
&=& \sum_{i=1}^n a_{ij} \underline{z}^{\underline{\alpha}} z_i z_j  \frac{\partial (a_{hk} \underline{z}^{\underline{\alpha}} z_h z_k)}{\partial z_i} + a_{ih} \underline{z}^{\underline{\alpha}} z_i z_h  \frac{\partial (a_{kj} \underline{z}^{\underline{\alpha}} z_k z_j)}{\partial z_i} +
a_{ik} \underline{z}^{\underline{\alpha}} z_i z_k  \frac{\partial (a_{jh} \underline{z}^{\underline{\alpha}} z_j z_h)}{\partial z_i} = \\
&=& \underline{z}^{\underline \alpha}  \underline{z}^{\underline \alpha} z_j z_h z_k \sum_{i=1}^n \alpha_i   \left(a_{ij} a_{hk} + a_{ih} a_{kj}  + a_{ik} a_{jh}\right).\end{eqnarray*}
Since the matrix $A=\left(a_{ij}\right)_{ij}$ is antisymmetric, if two of the indexes $i, j,h,k$ are equal the sum in the parenthesis vanishes and the statement is proved.
 \end{proof}

\section{Proof of the main theorem}\label{sec.main thm}
Let $N$ be a $n$-dimensional lattice, $M$ its dual lattice and $\Sigma$ in $N_\RR$.
Let $X=TV(\Sigma)$ be the normal toric variety associated to the fan $\Sigma$.
Recall that, $\Sigma$ is smooth if and only if $X$ is smooth and $\Sigma$ is complete if and only if $X$ is compact.
Let $\pi\in  \Gamma(X, \bigwedge^2 \Theta_X) $ be an equivariant bi-vector field on $X$.
We prove the following
\begin{maintheorem}
Let $X$ be a smooth toric variety of dimension $n$ and $\pi \in \Gamma(X, \bigwedge^2 \Theta_X)$ an equivariant bi-vector field on $X$. 
Consider an integer $k > 0$ such that $2k < n$. Then the degeneracy locus
$X_{\leq 2k}$ is not empty and contains at least a component of dimension $\geq 2k+1$.
Moreover, if the degeneracy locus $X_{\leq 0}$ is not empty, it contains at least a curve.
If in addition $X$ is compact, then $X_{\leq 0}\neq \emptyset$. 
\end{maintheorem}

\begin{proof}
Let $X=TV(\Sigma)$ as above.
Let $Y_{\sigma_0}$ be the affine subvariety of $X$ corresponding to a maximal cone $\sigma_0$ in the fan $\Sigma$. Since $\Sigma$ is smooth, we can assume without loss of generality that $\sigma_0$ is generated by the standard basis $\{e_1, \ldots, e_n\}$ of $\ZZ^n$.
On $Y_{\sigma_0}$ there are coordinates $z_i:=\chi^{e_i}$ and the equivariant bi-vector field can be written as
\[ \pi= \sum_{j > i=1}^n  2a_{ij} z_1^{\alpha_1} \cdots z_i^{\alpha_i+1} \cdots z_j^{\alpha_j+1}\cdots z_n^{\alpha_n} \frac{\partial}{\partial z_i} \wedge\frac{\partial}{\partial z_j}, \]
where  $A=\left( a_{ij}\right)_{ij}$ is an antisymmetric matrix with complex coefficients. Let $\Pi=(a_{ij} \cdot {\underline{z}}^{\underline{\alpha}} \cdot z_i \cdot z_j)_{ij}$ be the matrix whose rank defines the rank of the form $\pi$.  Note that, if $a_{ij} \neq 0$ for some $i$ and $j$, the exponents satisfy 
\[\alpha_i, \alpha_j \geq -1,\quad \mbox{ and} \quad \alpha_h \geq 0 \ \ \forall \ h\neq i,j.   \] 
First suppose $\alpha_i \geq 0$, for all $i\in \{ 1, \ldots , n\}$. Let $k \geq 0$ such that $2k <n$. Consider the $(2k+1)$-dimensional orbit $O$  defined on $Y_{\sigma_0}$ by:
\[ \{  (z_1, \ldots , z_n) \in Y_{\sigma_0} \mid z_i =0 \mbox{ for } 1 \leq i \leq n-2k -1 \mbox{ and }   z_i\neq0 \mbox{ for } n-2k \leq i \leq n \}.\]
As observed at the end of Subsection \ref{sec.orbit-cone}, this orbit corresponds to the face $\tau$ of $\sigma_0$ generated by $\{e_1, \ldots, e_{n-2k-1}\}$.
The matrix $\Pi$ calculated on a point of $O$ has at least $n-2k-1$ zero rows and $n-2k-1$ zero columns, thus its  rank is less or equal then $2k$ (since its rank is even). 
Thus, $O \subset X_{\leq2k}$ is the component of dimension $2k+1$ we are looking for.

Suppose now there exists $h_1,\ldots, h_s \in \{1,\ldots, n\}$ such that $\alpha_{h_i} =-1$ for all $i\in \{ 1,\ldots s\}$ and $\alpha_{k} \geq 0$ for all $k \not \in \{h_1, \ldots, h_s\}$. 
If $s=1$, the above arguments prove the theorem.
To be precise, one can choose any orbit $O$ such that the coordinate $z_{h_1} \neq 0$. 
If $s>2$, the matrix $A$ is the zero matrix and there is nothing to prove.
If $s=2$, the matrix $A$ has just two non-zero coefficients. To simplify the notation, assume $h_1=1$, $h_2=2$. Thus, $a_{12} =-a_{12}\neq 0$ are the only two non zero coefficients, the exponents are $\underline{\alpha}=(-1,-1, \alpha_3, \ldots, \alpha_n)$ with $\alpha_i \geq 0$ for all  $ i \geq 3$, and  on $Y_{\sigma_0}$ the form $\pi$ becomes 
\begin{equation} \label{eq.forma di rango minore di due} \pi= 2a_{12} z^{\alpha_3}_3 \cdots z^{\alpha_n}_n \frac{\partial}{\partial z_{1}} \wedge  \frac{\partial}{\partial z_{2}}. \end{equation} 
Thus, its rank is $\leq 2$ on the whole $Y_{\sigma_0}$. 
Let's analyse the degeneracy locus $X_{\leq 0}$.  
If there exists at least one $\alpha_i >0$ for $i \geq 3$, on the $(n-1)$-dimensional orbit  
$\{(z_1, \ldots, z_n) \in Y_{\sigma_0} \mid z_i=0\}$ the form $\pi$ has rank $0$. 

The remaining case is when $\pi$ is expressed on $Y_{\sigma_0}$ by
\begin{equation} \label{eq.forma di rango due} \pi=   2a_{12} \frac{\partial}{\partial z_{1}} \wedge  \frac{\partial}{\partial z_{2}} \quad \mbox{with } a_{12} \neq 0. \end{equation}
In the compact case, we have to prove that there exists a subvariety of dimension $\geq 1$ on which the form $\pi$ 
has rank $=0$.
Consider the hyperplane $H_0:=\{(z_1,\ldots, z_n) \mid z_1+z_2 =0\} \subset N_\RR$. Recall that $\Sigma$ is a fan, thus every cone is strongly convex, and it is complete. 
By Lemma  \ref{lemma.cono}, there exists a maximal cone $\sigma \in \Sigma$, such that at least one of its generating ray $\rho=(\rho_1, \ldots, \rho_n)$ lies in the halfspace $H_{0}^-$.
By lemma \ref{lemma.cambio di coordinate}, on the affine chart $Y_\sigma$ corresponding to the cone $\sigma\in \Sigma$ the form $\pi$ becomes
\[  \pi= \sum_{j > i=1}^n  2b_{ij} w_1^{\beta_1} \cdots w_i^{\beta_i+1} \cdots w_j^{\beta_j+1}\cdots w_n^{\beta_n} \frac{\partial}{\partial w_i} \wedge\frac{\partial}{\partial w_j}. \]
The coordinates $w_i$ are the one on $Y_\sigma$ corresponding to the basis of $\sigma$ given by its rays and $\underline{\beta} = R^t \cdot \underline{\alpha}$, 
where $R$ is the matrix of the base change between the basis given by the rays of $\sigma$ and the canonical basis of $\ZZ^n$. Thus, 
\[ \left(\begin{array}{c} \beta_1 \\ \vdots \\ \beta_n \end{array}\right) =R^t  \left(\begin{array}{c} \alpha_1 \\ \vdots \\ \alpha_n \end{array}\right)=\left( \begin{array}{ccc} 
 \rho_1 & \ldots & \rho_n \\
r_{12} & \ldots & r_{n2} \\
\vdots & \ldots & \vdots \\
r_{1n} & \ldots & r_{nn}  \end{array}\right)   \cdot \left(\begin{array}{c} -1 \\ -1 \\ 0 \\ \vdots\\ 0 \end{array}\right) = \left(\begin{array}{c} - \rho_1 -\rho_2 \\ \beta_2 \\ \vdots \\\beta_n \end{array}\right).\]
Since $\rho=(\rho_1, \ldots, \rho_n) \in H_0^-$, $-\rho_1 - \rho_2>0$, and all the other exponent $\beta_i$ are  $\geq -1$. On this chart, we can be proved the existence of the required subvariety, as discussed above.

If $X$ is not compact, i.e. the fan is not complete, the argument above does not hold. The degeneracy locus $X_{\leq 0}$ is not empty if and only if we can find an affine subvariety $Y_{\sigma}\subset X$ on which $\pi$ has the form (\ref{eq.forma di rango minore di due}) with at least one $\alpha_i >0$ and we can be proved the existence of the required subvariety, as discussed above.
\end{proof}

The proof shows quite clearly where compactness is needed in our theorem. Moreover, the hypothesis for $\pi$ to be equivariant is necessary at least for the non compact case. Let's see some examples.
\begin{example}
Consider the toric variety $\CC^n$. The degeneracy locus $X_{\leq 0}$ of any equivariant bi-vector field of the form
\[ \pi= a \frac{\partial}{\partial z_i}\wedge \frac{\partial}{\partial z_j}, \quad \mbox{with } a\neq 0 \mbox{ and } i, j \in \{ 1,\ldots , n\} \]
is empty, since the rank of $\pi$ is two everywhere. 
\end{example}
\begin{example}
This example is due to Gay \cite{Gay}.
On  the toric variety $\CC^3$, consider the bi-vector field
\[ \pi=  z_3 \frac{\partial}{\partial z_1}\wedge \frac{\partial}{\partial z_2} +  z_2 \frac{\partial}{\partial z_1}\wedge \frac{\partial}{\partial z_3} +   2 z_1 \frac{\partial}{\partial z_2}\wedge \frac{\partial}{\partial z_3}.  \]
It is immediate to see that $\pi$ is not equivariat. One can verify, it is a Poisson structure, i.e. it satisfies the condition $[\pi,\pi]_{SN}=0$. 
The rank of $\pi$ is the rank of the matrix
\[
\left( \begin{array}{ccc}0 & z_3 & z_2 \\ -z_3 & 0 & z_1 \\ -z_2 & -z_1 & 0 \end{array}\right).
\]
Observe that, the degeneracy locus $X_{\leq 0} = \{ 0 \}$.
\end{example}

\end{document}